\documentclass[11pt]{article}
\usepackage{amsmath,amssymb,amscd,latexsym,amsthm,mathrsfs,amsfonts}
\usepackage[unicode]{hyperref}
\usepackage{hypbmsec,longtable}
\usepackage{graphicx}
\usepackage{fancyhdr}
\ifx\pdfoutput\undefined\else
\hypersetup{
colorlinks=true,
linkcolor=blue,
citecolor=blue,
urlcolor=blue,
filecolor=blue,
bookmarksnumbered=true,
pdfstartview=FitH,
pdfhighlight=/N
}
\fi
\usepackage{pgf,tikz}

\usepackage{array}
\allowdisplaybreaks[1]

\newtheorem{thm}{Theorem}[section] 

\newtheorem{lemma}[thm]{Lemma} 

\newtheorem{prop}[thm]{Proposition} 

\theoremstyle{definition}

\newtheorem{defn}[thm]{Definition} 

\newtheoremstyle{TheoremNum}{\topsep}{\topsep}{\itshape}{}{\bfseries}{.}{ }{\thmname{#1}\thmnote{ \bfseries #3}}                                                       

\theoremstyle{TheoremNum}
\newtheorem{thmn}{Theorem}

\usepackage{cite}

\renewcommand{\author}[1]{\large\rm #1\\ \bigskip}

\renewcommand{\title}[1]{\bigskip\bigskip\Large\bf #1\bigskip\bigskip\\}

\begin{document}

\vglue .3 cm

\begin{center}

\title{Permutations sortable by deques and two stacks in parallel share the same growth rate.}

\author{Andrew Elvey Price\footnote[1]{email:andrew.elvey@univ-tours.fr\newline This research was supported by the European Research Council (ERC) in the European Union's Horizon 2020 research and innovation programme, under the Grant Agreement No. 759702.}
}

LaBRI, Universit\'e de Bordeaux, France.

\end{center}
\setcounter{footnote}{0}

\begin{abstract}

Recently Albert and Bousquet-M\'elou obtained the solution to the long-standing problem of the enumeration of permutations sortable by two stacks in parallel (2sip). Their solution was expressed in terms of functional equations. E.P. and Guttmann then showed that the equally long-standing problem of the number of permutations sortable by a double ended queue (deque) can be simply related to the solution of the same functional equations. They then conjectured that the radius of convergence of both generating functions is the same, and reduced this conjecture to a series of conjectures of Albert and Bousquet-M\'elou regarding a generating function for quarter-plane loops. In this note we prove that the two growth rates are equal, using a combinatorial argument on certain lattice paths which are in bijection with the two classes. As a corollary we prove that the generating function $P(t)$ for permutations sortable by two stacks in parallel satisfies an inequality which was conjectured by Albert and Bousquet-M\'elou.
\end{abstract}

  {


\vskip .2cm


\vskip .3cm

\section{Introduction}

In his seminal book {\em the art of computer programming} \cite{K68} Knuth considered a number of classical data structures from the perspective of the permutations which they could produce. Equivalently one can ask which permutations can be sorted by these machines as these will be exactly the inverses of the producible permutations. The first non-trivial machine that Knuth considered in this context is a single stack. In this case, starting with the identity permutation $12\cdots n$ in the input, two operations are permitted as shown in Figure \ref{fig:stack}:
\begin{itemize}
\item The input operation $I$, which moves the next element from the input onto the top of the stack.
\item The output operation $O$, which moves the top element of the stack to the output.
\end{itemize}
\begin{figure}[htbp]
   \centering
   \includegraphics[width=2.5in]{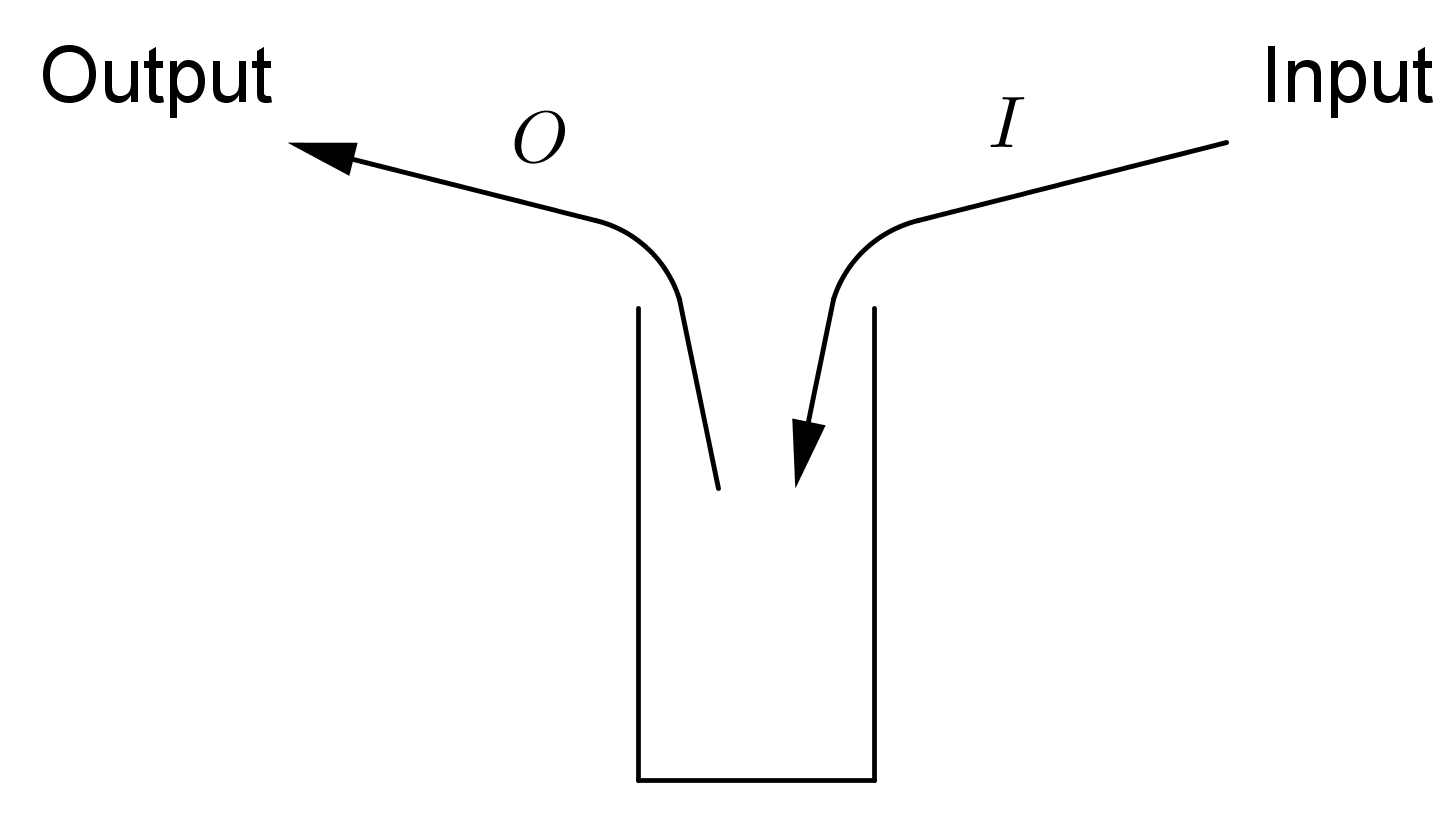} 
   \caption[The input and output operations on a stack.]{The input and output operations $I$ and $O$ on a stack.}
   \label{fig:stack}
\end{figure}
The permutation is defined by the order in which the elements are output. In particular, he showed that the number of permutations of length $n$ which can be produced by a single stack is the $n$th Catalan number
\[C_{n}=\frac{1}{n+1}{2n\choose n}.\] In this case, a valid sequence of the operations $I$ and $O$ corresponds to a unique Dyck path, with input operations corresponding to up steps and output operations corresponding to down steps. Moreover, Knuth showed that no two operation sequences produce the same permutation, so the number of permutations of size $n$ which can be produced by a single stack is equal to the number of Dyck paths of length $2n$, which is given by the $n$th Catalan number
\[C_{n}=\frac{1}{n+1}{2n\choose n}\sim \frac{1}{\sqrt{\pi}}4^{n}n^{-3/2}.\]
Further to this, Knuth characterised the set of achievable permutations, showing that they are exactly the permutations which are now called $312$-avoiding permutations. This is one of the first results in the field of pattern avoiding permutations.
Knuth then asked the same question for each of the following three more complicated data structures, shown in Figures \ref{fig:2sip}, \ref{fig:deque} and \ref{fig:2sis}:
\begin{itemize}
\item two stacks in parallel (2sip).
\item A double ended queue (deque).
\item two stacks in series (2sis).
\end{itemize}
For each of these machines, the set of sortable permutations has been shown to be a permutation class with infinite basis \cite{P73,M02}. In other words, the set of sortable permutations cannot be defined as those permutations which avoid some finite set of permutation patterns. For all three problems a polynomial time algorithm is known for determining whether a permutation is sortable \cite{EI71,PR13}.

\begin{figure}[ht]
\begin{minipage}{0.41\textwidth}
\includegraphics[width=0.97\linewidth]{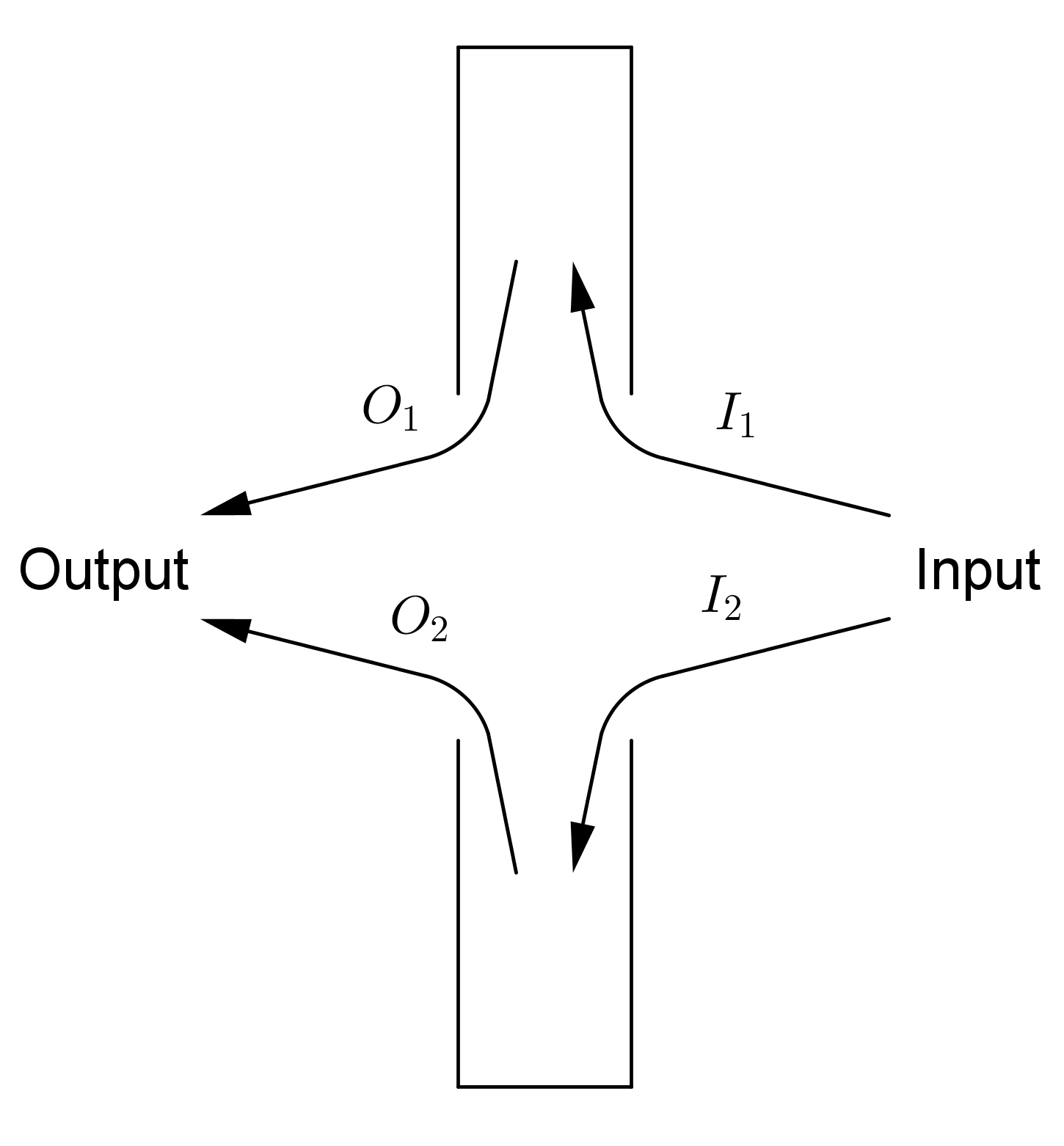} 
\caption{The input and output operations $I_{1}$, $I_{2}$, $O_{1}$ and $O_{2}$ on two stacks in parallel. }
\label{fig:2sip}
\end{minipage}\hfill
\begin{minipage}{0.55\textwidth}
\includegraphics[width=0.97\linewidth]{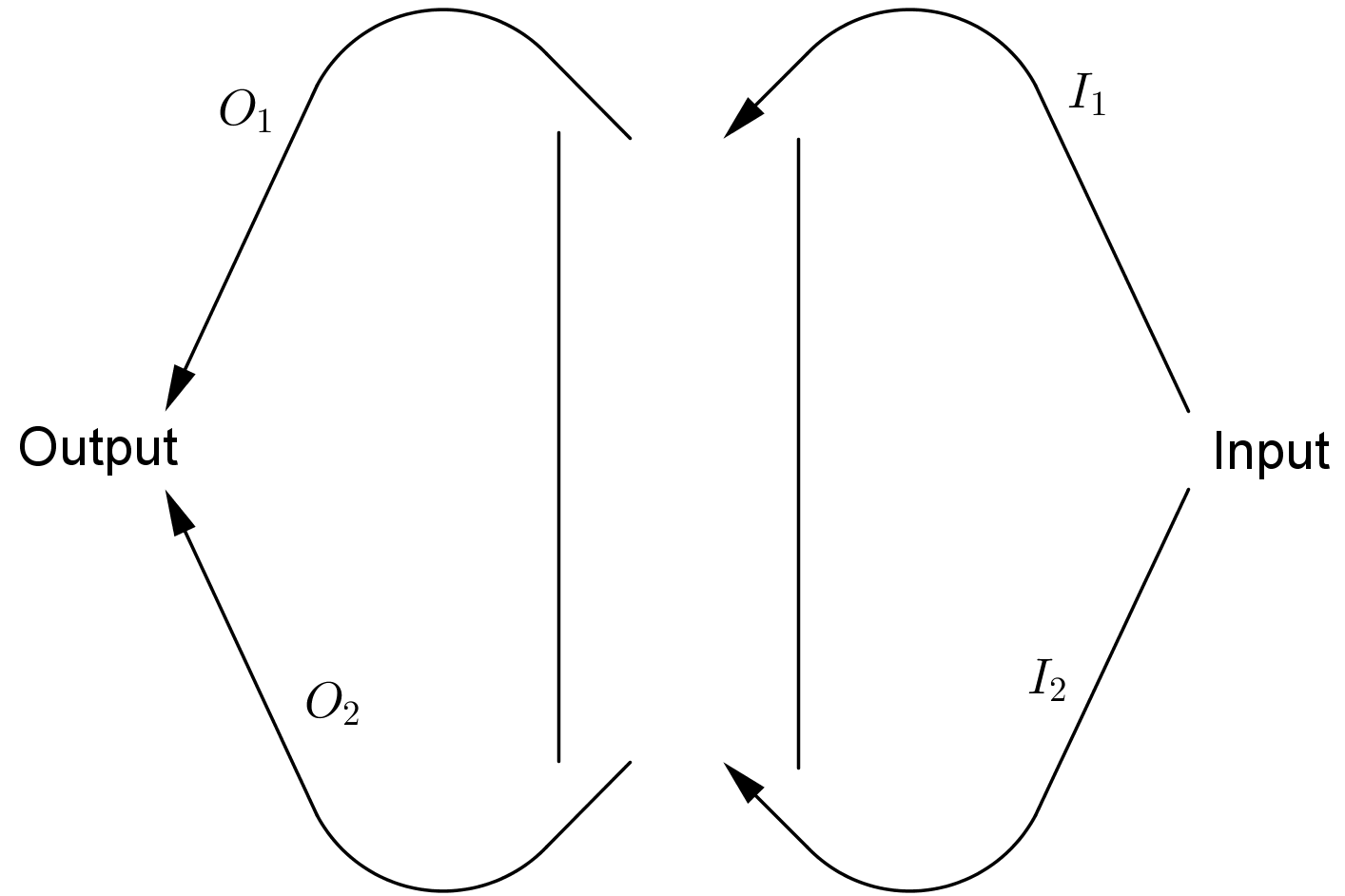} 
\caption{The input and output operations $I_{1}$, $I_{2}$, $O_{1}$ and $O_{2}$ on a deque.}
\label{fig:deque}
\end{minipage}
\end{figure}
\begin{figure}[htbp]
   \centering
   \includegraphics[width=3.3in]{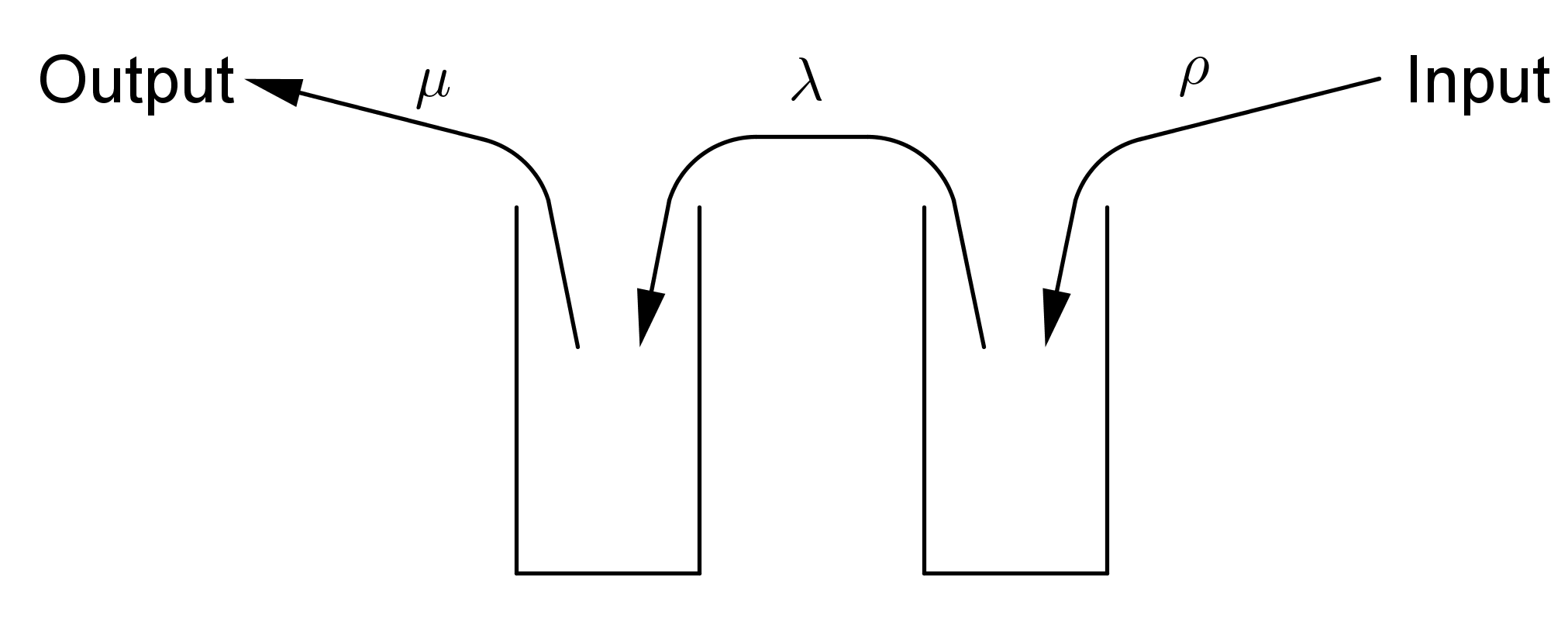} 
   \caption{The operations $\rho$, $\lambda$ and $\mu$ on two stacks in series.}
   \label{fig:2sis}
\end{figure}

We will focus on the problem of {\em enumerating} permutations sortable by one of these machines. Albert, Atkinson and Linton considered all three enumeration problems in \cite{AAL10}, where they derived upper and lower bounds on the growth rates $\mu_{p}$, $\mu_{s}$ and $\mu_{d}$ for 2sip-sortable, 2sis-sortable and deque-sortable permutations. They were also the first to suggest that the growth rates $\mu_{d}$ and $\mu_{p}$ may be equal based on the similarities in the bounds that they obtained (although they also suggested that these may both be equal to 8, which is now known to be incorrect \cite{EPPhD}). All of these problems remained open until 2015 when Albert and Bousquet-M\'elou enumerated permutations sortable by two stacks in parallel \cite{AB15}. Their solution took the form of a system of functional equations which characterise the generating function $P(t)$ of these permutations.  From an asymptotic perspective, these equations are surprisingly impenetrable, in that we are still unable to determine either the growth rate or the sub-exponential behaviour of the sequence exactly. In 2017, E.P. and Guttmann solved the problem for permutations sortable by a deque by demonstrating the following simple algebraic relationship between their generating function $D(t)$ and the generating function $P(t)$ of permutations sortable by two stacks in parallel:
\begin{equation}\label{DinP}2D(t)=2+t+2Pt-2Pt^2-t\sqrt{1-4P+4P^2-8P^2t+4P^2t^2-4Pt}.\end{equation}
They then reduced the problem of showing that $\mu_{d}=\mu_{p}$ to three conjectures of Albert and Bousquet-M\'elou regarding quarter-plane walks. Moreover, they used the equations from \cite{AB15} as well as \eqref{DinP} to compute over $1000$ coefficients of each generating function, which they analysed using the method of differential approximants. Using this analysis they gave estimates of $\mu_{p}$ and $\mu_{d}$ which agreed to 10 significant digits, further supporting their conjecture that these growth rates are equal. This is the conjecture that we prove in this article:
\begin{thmn}[\ref{mainthm}] The growth rate $\mu_{d}$ of deque-sortable permutations is equal to the growth rate $\mu_{p}$ of 2sip-sortable permutations.\end{thmn}
We were not able to prove the three conjectures of Albert and Bousquet-M\'elou about quarter-plane walks. Instead our proof of this theorem uses descriptions of $D(t)$ and $P(t)$ as generating functions for certain explicit classes of walks given in \cite{AB15} and \cite{EG17_deque}.

The outline of this paper is as follows: In Section \ref{growth rate equality} we will prove Theorem \ref{mainthm}. In Section \ref{2sipgf} we will use this result in conjunction with \eqref{DinP} to prove the inequality
\[\sqrt{2P(t)}\geq1+\sqrt{2tP(t)},\]
for $t\in[0,1/\mu_{p}]$, which was conjectured in \cite{AB15}. Finally, in the Section \ref{bounds} we will describe a new rigorous lower bound on the shared growth rate $\mu_{d}=\mu_{p}$. To the author's knowledge, this is now also the strongest known lower bound for the growth rate of permutations sortable by two stacks in {\em series}.

\section{The growth rates of deque-sortable and 2sip-sortable permutations}
\label{growth rate equality}
In this section we prove Theorem \ref{mainthm}. First we will describe a class of lattice walks counted by each of the two generating functions $P(t)$ and $D(t)$. These were essentially described in \cite{AB15} and \cite{EG17_deque}, respectively, albeit in the language of operation sequences rather than walks. The correspondence can be seen by replacing each operation $I_{1},I_{2},O_{1},O_{2}$ with the step $N,E,S,W$, respectively. In each case the walks start from the origin and have steps from the set $\{N,E,S,W\}$. In the case of $P(t)$, the walks are restricted to the quarter-plane $\{(x,y)|x,y\geq0\}$ while for $D(t)$, the walks are only restricted to the diagonal half-plane $\{(x,y)|x+y\geq0\}$. Before defining the two types of walks, we need to define three properties of these walks.

\begin{defn} A walk is called {\em eager} if it contains no $ES$ or $NW$ corners.\end{defn}

\begin{defn} A {\em QP-subloop} of a walk is a subsection of the walk which starts and ends at the same point $(p,q)$ and only visits points $(x,y)$ satisfying $x\geq p$ and $y\geq q$. In other words, a QP-subloop is a translate of a quarter-plane excursion. A walk $w$ is called {\em standard} if every QP-subloop of $w$ begins with an $N$ step.\end{defn}

\begin{defn} A walk is {\em vertical-happy} if each step from the lines $x+y=0$ and $x+y=1$ is either $N$ or $S$.\end{defn}
In the language of \cite{AB15} and \cite{EG17_deque} an eager walk corresponds to an operation sequence that outputs eagerly, a standard walk corresponds to a standard operation sequence, a QP-subloop corresponds to a tsip-subword of the corresponding operation sequence and a verticle-happy walk corresponds to a top-happy operation sequence. Note that the terms {\em top-happy} and {\em tsip-subword} were only used in \cite{EG17_deque} while the terms {\em standard} and {\em outputs eagerly} were introduced in \cite{AB15} and used in both articles.

Now we are ready to define the walks which are counted by $P(t)$ and $D(t)$, which we will call 2sip-walks and deque-walks respectively.

\begin{defn} A {\em 2sip-walk} is a walk confined to the quarter-plane, which starts and ends at the point $(0,0)$ and which is standard and eager.\end{defn}

\begin{defn} A {\em deque-walk} is a walk confined to the diagonal half-plane $\{(x,y)|x+y\geq0\}$, which starts at the point $(0,0)$, ends at some point $(x,-x)$ and which is standard, eager and vertical happy.\end{defn}

The following two propositions state the desired results that 2sip-walks correspond bijectively with 2sip-sortable permutations and deque-walks correspond bijectively with deque-sortable permutations. It follows immediately from these propositions that the series $P(t)$ counts 2sip-walks by half-length while the series $D(t)$ counts deque-walks by half-length. These are essentially \cite[Prop.~6]{AB15} and \cite[Prop.~2.5]{EG17_deque} respectively, so we omit the proofs here.
\begin{prop} For each integer $n\geq0$, each 2sip-walk of length $2n$ corresponds to a unique 2sip-sortable permutation of length $n$. In particular, the number of these walks is equal to $p_{n}$, the number of 2sip-sortable permutations of length $n$.\end{prop}
\begin{prop} For each integer $n\geq0$, each deque-walk of length $2n$ corresponds to a unique deque-sortable permutation of length $n$. In particular, the number of these walks is equal to $d_{n}$, the number of deque-sortable permutations of length $n$.\end{prop}

It will be convenient to define a third type of walk, which simultaneously generalises 2sip-walks and deque-walks:
\begin{defn} A {\em big-walk} is a walk confined to the diagonal half-plane $\{(x,y)|x+y\geq0\}$, which starts at the point $(0,0)$, ends at some point $(x,-x)$ and which is standard and eager. Let $b_{n}$ be the number of these walks of length $2n$ and let
\[B(t)=\sum_{n=0}^{\infty}b_{n}t^n\]
be the series which counts big-walks by half-length.\end{defn}
Note that we can define deque-walks and 2sip-walks very quickly in terms of big-walks: a deque-walk is a vertical-happy big-walk, while a 2sip-walk is a big-walk which is confined to the quarter-plane. We note that concatenating two big-walks forms another big-walk, so we have the super multiplicative property $b_{m+n}\geq b_{m}b_{n}$. Together with Fekete's lemma \cite{Fekete}, and the easy upper bound $b_{m}\leq 16^m$, we see that the growth rate
\[\mu_{b}=\lim_{n\to\infty}\sqrt[n]{b_{n}}\]
exists and is finite. Completely analogous arguments show that the growth rates $\mu_{p}$ and $\mu_{d}$ exist. Before we prove the main theorem, we will define two sub-classes of big paths, which we call left-walks and right-walks.
\begin{defn}
A {\em left-walk} is a big-walk that ends at some point $(x,-x)$ satisfying $x\leq0$. Conversely, a {\em right-walk} is a big-walk that ends at some point $(x,-x)$ satisfying $x\geq0$.
\end{defn}
The following lemma states that there is a bijection between left-walks and right-walks of the same length. This justifies our introduction of big-walks, as the analogous symmetry does not hold for deque-walks.
\begin{lemma} For each $n$, the number of left-walks of length $2n$ is equal to the number of right-walks of length $2n$.\end{lemma}
\begin{proof}We will describe a procedure which bijectively converts left-walks $w_{l}$ into a right-walks $w_{r}$. Starting with a left-walk $w_{l}$, let $w_{l}'$ be the reflection of $w_{l}$ in the line $x=y$. Now each 2sip-loop in $w_{l}'$ begins with $E$. For each such 2sip-loop which is not contained in any other 2sip-loop, replace each $N$, $E$, $S$ and $W$ step with $E$, $N$, $W$ and $S$ respectively. Call the resulting path $w_{r}$. Then each 2sip-loop in $w_{r}$ will start with $N$, so $w_{r}$ is a big path. Moreover, the endpoint of $w_{r}$ will be the same as that of $w_{l}'$, so $w_{r}$ is a right-walk. We can then apply the exact same transformation to $w_{r}$ to reproduce $w_{l}$, so this transformation is bijective. Hence the number of left-walks is equal to the number of right-walks.\end{proof}
Let $\hat{b}_{n}$ be the number of left-walks of length $2n$, which, as we have just seen, is also the number of right-walks of length $2n$. Since each big-walk is either a right-walk or left-walk (or both), we have the inequality $2\hat{b}_{n}\geq b_{n}$. In particular this means that the growth rate
\[\mu_{b}=\lim_{n\to\infty}\sqrt[n]{\hat{b}_{n}}.\]

We are now ready to prove our main theorem:
\begin{thm}\label{mainthm}The exponential growth rates
\[\mu_{p}=\lim_{n\to\infty}\sqrt[n]{p_{n}},~~~~~~~~\mu_{d}=\lim_{n\to\infty}\sqrt[n]{d_{n}}~~~~~~~\text{and}~~~~~~~\mu_{b}=\lim_{n\to\infty}\sqrt[n]{b_{n}}\]
are all equal.\end{thm}
\begin{proof}
We already know that for each $n$, the inequality $p_{n}\leq d_{n}$ holds, since any 2sip-sortable permutation can also be sorted by a deque. Moreover, since every deque-walk is also a big-walk, we have the inequality $d_{n}\leq b_{n}$. Hence, it suffices to show that the growth rate $\mu_{b}$ of big-walks is no greater than the growth rate $\mu_{p}$ of 2sip-walks. In order to show that $\mu_{b}\leq\mu_{p}$, we will construct a 2sip-walk of length $2m^2+8m+4$ as follows:
\begin{itemize}
\item [(1)] Start with $2m+1$ north steps followed by $2m$ east steps then a north step followed by two south steps.
\item [(2)] Add $m$ big-walks as follows: if we are currently at a point $(2m+x,2m-x)$ with $x>0$, add a left-walk of length $2m$, and otherwise add a right-walk of length $2m$.
\item [(3)] Finally, add south steps until reaching the $x$-axis then west steps until reaching $(0,0)$.
\end{itemize}
An example of this path for $m=3$ is shown in Figure \ref{fig:long_path}.
For each of the $m$ big-walks added in stage (2), one has $\hat{b}_{m}$ choices, so the number of possible paths is $\hat{b}_{m}^{m}$. Note also that from the entire path we can reconstruct the $m$ big-walks simply removing the first $4m+2$ and last $4m$ steps, then separating the remaining $2m^2$ steps into $m$ pieces of $2m$ consecutive steps. These $m$ pieces will then be the $m$ big paths. Hence the $\hat{b}_{m}^{m}$ possible paths are all distinct.
\begin{figure}[htbp]
   \centering
   \includegraphics[width=3in]{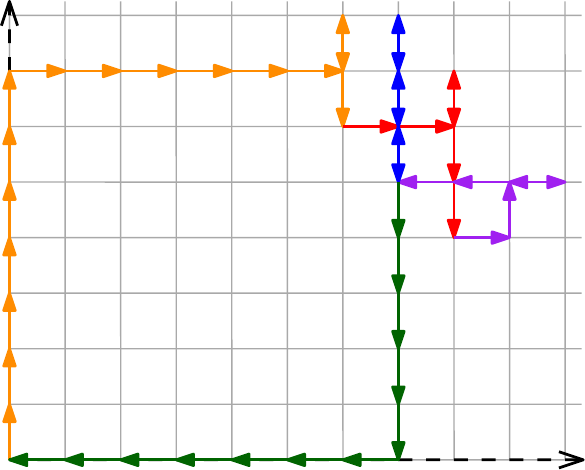} 
   \caption{An example of a 2sip-walk constructed in the proof of Theorem \ref{mainthm} for $m=3$. This path has length $2m^2+8m+4=46$. The section added in stage (1) is shown in orange, the three big-walks added in stage (2) are shown in red, purple and blue, respectively and the section added in stage (3) is green.}
   \label{fig:long_path}
\end{figure}

We will now show that each path is a 2sip-walk as claimed. First note that after stage (1), the path is at $(2m,2m)$. It is easy to show by induction that before and after each big-walk is added in stage (2), the path will be at some point $(2m+x,2m-x)$, with $x\in[-m,m]$. Finally, stage (3) takes $4m$ steps to return to $(0,0)$. At no point in this process is it possible that the path leaves the quarter-plane, hence the path is a quarter-plane walk which ends at $(0,0)$. We also note that the path cannot contain any ES or NW corners, so it is eager. Finally we will show that the path is standard. Suppose for the sake of contradiction that the path is not standard, so it contains a QP-subloop $\ell$ starting with E. Without loss of generality, $\ell$ does not return to its starting point until it ends, as otherwise it would contain a shorter QP-subloop, starting with the same step E. The QP-subloop $\ell$ cannot start before the first big-path, as the only E steps in stage (1) are at height $y=2m+1$, but the path goes to the point $(2m,2m)$ before it could return to the starting point. Since there are no E steps in stage (3), the QP-subloop $\ell$ must start in one of the big-walks, hence it must start and end at some point $(x_{\ell},y_{\ell})$ satisfying $x_{\ell}+y_{\ell}\geq 4m$. Since each big-walk is standard, $\ell$ cannot be contained entirely within one of these big-walks, so an end point $(2m+x,2m-x)$ of one of the big-paths must be an interior point of $\ell$. Since $\ell$ is a QP-subloop, we have $x_{\ell}\leq(2m+x)$ and $y_{\ell}\leq(2m-x)$. Moreover, these cannot both be equalities as we assumed that $\ell$ does not contain its starting point except at its ends. Hence $x_{\ell}+y_{\ell}<(2m+x)+(2m-x)=4m$. This is a contradiction, so in fact the entire path is standard. Hence the path is a 2sip-walk.

Therefore these $\hat{b}_{m}^{m}$ paths defined by the procedure are all distinct 2sip-walks. Hence, the number $p_{m^2+4m+1}$ of 2sip-walks of length $2m^2+8m+2$ satisfies the inequality
\[p_{m^2+4m+1}\geq\hat{b}_{m}^{m}.\]
Raising both sides of this inequality to the power $1/m^2$, then taking the limit $m\to\infty$ yields the desired inequality $\mu_{p}\geq\mu_{b}$. Since we also know that $\mu_{b}\geq\mu_{d}\geq\mu_{p}$, it follows that $\mu_{p}=\mu_{d}=\mu_{b}$, as required.
\end{proof}

\section{The generating function for 2sip-sortable permutations}
\label{2sipgf}
In \cite[Thm. 23]{AB15}, Albert and Bousquet-M\'elou proved, subject to two conjectures about quarter-plane walks, that
\[\frac{t}{(2-1/P(t))^2}\leq \rho_{Q}(1/P(t)-1),\]
for $0\leq t\leq t_{c}=1/\mu_{p}$, with equality if and only if $t=t_{c}$. In the equation above, $\rho_{Q}(a)$ denotes the radius of convergence of a certain class of weighted quarter-plane walks, with weight $a$. Combining this with their conjectured formula \cite[Conj. 11]{AB15} for $\rho_{Q}(a)$ then rearranging yields the conjecture that
\[\sqrt{2P(t)}\geq1+\sqrt{2tP(t)},\]
for $0\leq t\leq t_{c}=1/\mu$, with equality if and only if $t=t_{c}$. In this section we prove this inequality using our result that $\mu_{p}=\mu_{d}$, although we are not able to prove that equality holds at $t=t_{c}$.

Recall that the generating functions $P(t)$ and $D(t)$ are related by \eqref{DinP}, which we repeat here:
\[2D(t)=2+t+2Pt-2Pt^2-t\sqrt{1-4P+4P^2-8P^2t+4P^2t^2-4Pt}.\]
Using this equation in conjunction with the result that $\mu_{d}=\mu_{p}$, we can prove the main theorem of this section:
\begin{thm}The generating function $P(t)$ for 2sip-sortable permutations counted by length satisfies the inequality
\[\sqrt{2P(t)}\geq1+\sqrt{2tP(t)},\]
for $t\in[0,t_{c}]$, where $t_{c}$ is the radius of convergence of $P(t)$.\end{thm}
\begin{proof}
By Theorem \ref{mainthm}, we know that $t_{c}$ is also the radius of convergence of $D(t)$, so for $t\in[0,t_{c})$, we have $D(t)\in\mathbb{R}$. In particular, using \eqref{DinP}, this implies that
\[1-4P(t)+4P(t)^2-8P(t)^2t+4P(t)^2t^2-4P(t)t\geq0.\]
Factorising the left hand side and writing $P=P(t)$, we deduce that the product
\[\left(\sqrt{2P}-1-\sqrt{2tP}\right)\left(\sqrt{2P}+1-\sqrt{2tP}\right)\left(\sqrt{2P}-1+\sqrt{2tP}\right)\left(\sqrt{2P}+1+\sqrt{2tP}\right)\]
is non-negative. Since $t\in[0,1)$ and $P=P(t)\geq P(0)=1$, the second, third and fourth factors in the above expression are positive, so \[\sqrt{2P(t)}\geq1+\sqrt{2tP(t)},\]
as required. So far we have proved this for $t\in[0,t_{c})$, but taking the limit as $t\to t_{c}$ implies that the inequality also holds at $t=t_{c}$.\end{proof}

\section{Bounds on the growth rate}\label{bounds}
In this section we derive a new lower bound of $8.21927$ on the growth rate $\mu$ of permutations sortable by two stacks in parallel and permutations sortable by a deque. This improves on the bounds shown by Albert, Atkinson and Linton \cite{AAL10}  of $7.535$ for 2sip-sortable permutations and $7.890$ for deque-sortable permutations. Indeed, this even improves on their lower bound $8.156$ for the growth rate of permutations sortable by two stacks in series, since any permutation which can be sorted by two stacks in parallel can be sorted by two stacks in series. To derive our new bounds, we use the fact that the class of deque-sortable permutations is sum-closed, that is, if $\tau=\tau_{1}\tau_{2}\cdots\tau_{m}$ and $\pi=\pi_{1}\pi_{2}\cdots\pi_{n}$ are deque-sortable permutations (written in series form), then so is their direct sum $\tau\oplus\pi=\tau_{1}\tau_{2}\cdots\tau_{m}(\pi_{1}+m)(\pi_{2}+m)\cdots(\pi_{n}+m)$. Equivalently, concatenating two deque-walks yields another deque-walk. It follows that we can write
\[D(t)=\frac{1}{1-V(t)},\]
where $V(t)$ is the generating function for {\em primitive} deque-sortable permutations (or equivalently, for deque-walks that only visit the line $x+y=0$ at their endpoints). Recall that we computed the first $N=1336$ coefficients of the generating functions $D(t)$ of deque-sortable permutations \cite{EG17_deque}. Using these coefficients, we can compute that first $N$ coefficients of the generating function $V(t)$ of primitive deque-sortable permutations, using the equation above. The degree $N$ polynomial
\[V_{\leq N}(t)=\sum_{n=0}^{N}t^n([t^n]V(t))\]
satisfies $V_{\leq N}(t)\leq V(t)$ as $V(t)$ has non-negative coefficients. Hence,
\[D(t)\geq\frac{1}{1-V_{\leq N}(t)}.\]
For $N=1336$, the radius of convergence of the right hand side of the inequality is $t_{c}\approx0.1217$, so $1/t_{c}=8.21927\ldots$ is a lower bound for the growth rate $\mu$ of deque-sortable permutations.

Albert, Atkinson and Linton \cite{AAL10} also gave an upper bounds of $8.346$ for the growth rate of 2sip-sortable permutations and $8.352$ for deque-sortable permutations. Hence, using Theorem \ref{mainthm} this implies that the shared growth rate $\mu$ for these two classes lies in the interval $(8.219,8.346)$. These are to be compared to the estimated value $\mu\approx 8.281402207$ given by E.P. and Guttmann \cite{EG17_deque}. By contrast, the growth rate $\mu_{s}$ for 2sis-sortable permutations is now known to lie in the interval $(8.219,13.374)$ \cite{AAL10}, both of which are a long way from the estimate $\mu_{s}\approx 12.5$ given in \cite{EG17_2sis}.

\section*{Acknowledgments}
I would like to thank Anthony J Guttmann and Michael Wallner for carefully reading early versions of this article. I would also like to thank Robert Brignall for making a suggestion after my talk at pp2017 which led to my proving the main result of this article.

\end{document}